\documentclass[12pt]{article}
\usepackage{amsthm, amsmath, amssymb, graphicx,xcolor} 
\usepackage{ytableau, chessfss}
\usepackage{array}
\usepackage{hyperref}
\usepackage[margin = 1in]{geometry}

\usepackage[backend=biber, style=alphabetic, url=true, doi=true, eprint=true]{biblatex}
\usepackage{authblk}

\newcommand{\Rook}{\raisebox{-7pt}{\BlackRookOnWhite}}
\newcommand{\ZZ}{\mathbb{Z}}
\newcommand{\op}{\operatorname}
\newcommand{\inv}{\op{inv}}

\newcommand{\defn}{\textbf}

\newtheorem{proposition}{Proposition}
\theoremstyle{definition}
\newtheorem{remark}[proposition]{Remark}
\newtheorem{example}[proposition]{Example}

\title{Garsia--Remmel $q$-rook numbers \\ are not always unimodal}
\author[1]{Joel Brewster Lewis}
\author[2]{Alejandro H. Morales\footnote{corresponding author}}
\affil[1]{Department of Mathematics, George Washington University\\ \texttt{jblewis@gwu.edu}}
\affil[2]{LACIM, D\'epartement de Mathématiques, Universit\'e du Qu\'ebec \`a Montr\'eal\\ \texttt{ahmorales.math@gmail.com}} 

\date{September 2025}

\begin{document}

\maketitle

\begin{center}
{\em Dedicated to the memory of Adriano Garsia.}
\end{center}

\begin{abstract}
We show by an explicit example that the Garsia--Remmel $q$-rook numbers of Ferrers boards do not all have unimodal sequences of coefficients.  This resolves in the negative a question from 1986 by the aforementioned authors. 
\end{abstract}

\begin{center}
{\bf keywords:} $q$-rook numbers, rook placements, Ferrers boards, unimodality, counterexamples, q-Stirling numbers
\end{center}

\bigskip

The central objects in rook theory are the \defn{placements of (nonattacking) rooks} on a \defn{board}.  Formally, a board $B$ is a finite subset of $\ZZ_{> 0} \times \ZZ_{> 0}$, and a rook placement on $B$ is a subset $w$ of $B$ such that no two elements of $w$ agree on either coordinate.  The terminology comes from the natural visual representation of these objects, as in Figure~\ref{fig:a rook placement}.  The \defn{rook number} $R_r(B)$ is defined as the number of placements of $r$ rooks on the board $B$.

\begin{figure}[hbt]
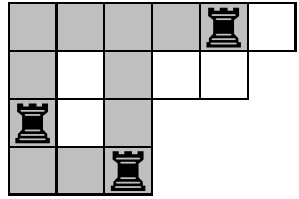

   \[ \begin{ytableau}
*(lightgray) & *(lightgray)& *(lightgray)& *(lightgray)& *(lightgray)\Rook& \\
*(lightgray)& &*(lightgray) & & \\
*(lightgray)\Rook & &*(lightgray) \\
*(lightgray)&*(lightgray) &*(lightgray) \Rook \\
\end{ytableau}
\]
    \caption{The rook placement $\{(1, 5), (3, 1), (4, 3)\}$ on the board $B_{\langle 6, 5, 3, 3\rangle}$ has inversion number $5$.  (Here we use matrix coordinates.)}
    \label{fig:a rook placement}
\end{figure}

A case of particular interest (beginning in the first systematic study \cite{KaplanskyRiordan} of the field) is when $B = B_\lambda$ is a \defn{Ferrers board} associated to an \defn{integer partition} $\lambda$: that is, $\lambda = \langle \lambda_1, \ldots, \lambda_\ell\rangle$ is a finite nonincreasing sequence of positive integers and $B_\lambda := \{ (i, j) : 1 \leq i \leq \ell, 1 \leq j \leq \lambda_i\}$.  For example, in the case of the \defn{staircase partition} $\delta_n := \langle n - 1, n-2, \ldots, 2, 1\rangle$, one has that
\[
R_k(B_{\delta_{n}}) = S(n, n - k)
\]
is a Stirling number of the second kind \cite[\S 5]{KaplanskyRiordan}.  In general, two partitions $\lambda, \mu$ are said to be \defn{rook equivalent} if $R_k(B_\lambda) = R_k(B_\mu)$ for all $k$ \cite{GJW1}. This condition can be phrased in various ways; for our purposes, the most convenient characterization \cite[Cor.~3.2]{cotardo2023diagonalsferrersdiagram} is that $\lambda$ and $\mu$ are rook equivalent if and only if $d_m(\lambda) = d_m(\mu)$ for all $m$, where the $d_m$ are the \defn{diagonal counts}
\begin{equation}\label{eq:diagonals}
    d_m(\nu) := | \{(i, j) \in B_\nu : i + j = m + 1\}|.
\end{equation}

In \cite{GarsiaRemmel}, Garsia and Remmel introduced a $q$-analogue of the rook number.  This analogue is in terms of the \defn{inversion number}, defined as follows: given a rook placement $w$ on a Ferrers board $B_\lambda$, for each rook $(a, b) \in w$, delete from $B_\lambda$ all elements $(i, b)$ with $i \leq a$ and $(a, j)$ with $j \leq b$; the inversion number $\inv(w)$ is the number of undeleted cells.  Then the \defn{(Garsia--Remmel) $q$-rook number} of \defn{rank} $k$ is
\[
R_k(\lambda; q) := \sum_{\substack{w \text{ a placement of} \\ \text{$k$ rooks on } B_\lambda}} q^{\inv(w)}.
\]
By definition, $R_k(\lambda; q)$ is a polynomial in the formal variable $q$ with nonnegative integer coefficients such that $R_k(\lambda; 1) = R_k(B_\lambda)$.  For example, the rank-$3$ $q$-rook number of the partition $\langle 6, 5, 3, 3\rangle$ is $q^{11} + 4q^{10} + 10q^9 + 19q^8 + 27q^7 + 30q^6 + 25q^5 + 15q^4 + 6q^3 + q^2$, and the rook placement in Figure~\ref{fig:a rook placement} is one of the $25$ that contribute to the coefficient of $q^5$.
  
Analogously to the case $q = 1$, one says two partitions are \defn{$q$-rook equivalent} if they have the same $q$-rook numbers. Surprisingly, two partitions are rook equivalent if and only if they are $q$-rook equivalent \cite[p.~257]{GarsiaRemmel}. In the particular case $\lambda = \delta_n$, the polynomial $R_k(\lambda; q)$ is a \emph{$q$-Stirling number of the second kind} %$S_q(n,n-k)$  
\cite[Eq.~I.9]{GarsiaRemmel}%, defined by 
%\begin{equation} \label{eq:recurrence qStirling}
%S_q(n,k) = q^{k-1} S_q(n-1,k-1) + [k] S_q(n-1,k), \quad S_{q}(0,0) =1, \quad S_q(n,k)=0, \text{ otherwise.}
%\end{equation}
%F
; for more on these numbers, see \cite{MILNE1982173,WACHS199127} and Remarks~\ref{rem:q stir 1} and~\ref{rem:q stir 2} below.

The same polynomials $R_k(\lambda; q)$ also arise in enumerative linear algebra, when counting matrices of rank $k$ over a finite field whose nonzero entries all belong to $B_\lambda$ \cite[Thm.~1]{Haglund}; this connection has led to applications in coding theory, since up to a power of $q$, the $R_k(\lambda; q)$ give the weights of certain rank-metric codes (e.g., see \cite{Ravagnani,Gluesing-Luerssen_Ravagnani,ALCO_2024__7_2_555_0}). The $q$-rook numbers and the related {\em $q$-hit numbers} 
%(which in contrast \jbl{We have not introduced unimodality yet!  Would it be ok if we removed this parenthetical?} are symmetric and unimodal \cite[Thm.~6]{haglund2024alphachromaticsymmetricfunctions}) 
have also recently appeared in the expansions of {\em chromatic symmetric functions} of graphs corresponding to certain Dyck paths called {\em abelian} \cite{ABREU2021105407,COLMENAREJO2023103595,nadeau2021qdeformationalgebraklyachkomacdonalds,haglund2024alphachromaticsymmetricfunctions}.

Based on an assortment of empirical and circumstantial evidence, Garsia and Remmel conjectured \cite[p.~250]{GarsiaRemmel} that the polynomials $R_k(\lambda; q)$ always have \defn{unimodal sequence of coefficients}, i.e., that if
\[
R_k(\lambda; q) = a_N q^N + a_{N - 1} q^{N - 1} + \ldots + a_1 q + a_0
\]
then for some $m$ it holds that $a_N \leq a_{N - 1} \leq \cdots \leq a_{m + 1} \leq a_m \geq a_{m  - 1} \geq \cdots \geq a_1 \geq a_0$.
As we now show, it is not difficult to prove the Garsia--Remmel unimodality conjecture for the cases of the maximum and minimum nontrivial numbers of rooks.
\begin{proposition} \label{prop:full rank}
For every partition $\lambda = \langle \lambda_1, \ldots, \lambda_\ell\rangle$, $R_{\ell}(\lambda; q)$ is unimodal.
\end{proposition}
\begin{proof}
In the case of $\ell$ rooks, there must be one rook in each row.  Consider iteratively placing the rooks beginning with the last row.  Regardless of how one has placed rooks in rows $k + 1, \ldots, \ell$, there will be $\lambda_k - (\ell - k)$ legal positions to place a rook in row $k$, and the different choices will leave $0, 1, \ldots, \lambda_k - (\ell - k) - 1$ inversions in row $k$.  Consequently
\begin{equation}\label{eq:full rank}
R_{\ell}(\lambda; q) = [\lambda_{\ell}]_q \cdot [\lambda_{\ell - 1} - 1]_q \cdots [\lambda_1 - (\ell - 1)]_q 
\end{equation}
where $[n]_q := 1 + q + q^2 + \ldots + q^{n - 1}$. Thus $R_\ell(\lambda; q)$ is a product of log-concave polynomials (i.e., polynomials for which the coefficient sequence satisfies $a_i^2 \geq a_{i - 1}a_{i + 1}$) with nonnegative coefficients and no internal zeros.  The family of such polynomials is closed under products \cite[Prop.~2]{StanleyLogConcave}, and furthermore every such polynomial is unimodal \cite[p.~500]{StanleyLogConcave}.
\end{proof}

\begin{proposition}
For every partition $\lambda$, $R_1(\lambda;q)$ is unimodal.
\end{proposition}
\begin{proof}
It is easy to see \cite[proof of Cor.~3.2]{cotardo2023diagonalsferrersdiagram} that 
\begin{equation}\label{eq:rank 1}
R_1(\lambda; q) = \sum_{i} d_i(\lambda) \cdot q^{|\lambda| - i}
\end{equation}
where $d_i$ is as in \eqref{eq:diagonals} and $|\lambda|=\lambda_1+\lambda_2+\cdots + \lambda_{\ell}$ is the \defn{size} of $\lambda$.
This polynomial is unimodal because, taking $m$ to be the maximum index for which $d_m = m$ and fixing a choice of a box $(a, m + 2 - a)$ that does \emph{not} belong to the $(m + 1)$st diagonal of $B_\lambda$, we have that the map 
\[
(x, y) \mapsto \begin{cases} (x, y - 1) & \text{ if } x < a, \\ (x - 1, y) & \text{ if } x > a
\end{cases}
\]
is an injection from the $k$th diagonal of $B_\lambda$ to the $(k - 1)$st for all $k > m$, and hence $d_m \geq d_{m + 1} \geq \cdots$.
\end{proof}

Progress on the general conjecture has been limited.  One reason for this is that, although $q$-rook numbers of maximum rank are always log-concave (per the proof of Proposition~\ref{prop:full rank}), this is not the case for lower-rank $q$-rook numbers, as shown in the next example.  Thus, the assorted machinery that has been built up (especially in the past decade \cite{AHK,BH,AGSVI,ChanPak}) to prove log-concavity results is not directly useful here.
\begin{example}
One has (either from \eqref{eq:rank 1} or by direct examination) that 
\[
R_1(\langle 4,1\rangle;q) = q^4+2q^3+q^2+q, 
\]
which is not log-concave (since $a_2^2 = 1 < 2 = a_1 \cdot a_3$). 
% More generally, one has that 
% \[
% R_1(\langle \ell,1\rangle;q) = q + q^2 + q^3 + \ldots + q^{\ell - 2} + 2q^{\ell - 1} + q^{\ell}.
% \]
\end{example}

A second reason that progress on the general conjecture has been limited is that the conjecture is false.
\begin{proposition}\label{main prop}
The $q$-rook number
\[
R_2(\langle 10, 9, 3, 2, 1\rangle; q)
\]
is not unimodal.
\end{proposition}
\begin{proof}
By considering whether or not the cell $(\ell, \lambda_\ell)$ has a rook, it is easy to see that the $q$-rook numbers satisfy the recurrence relation
\begin{multline}\label{eq:recurrence}
R_k(\langle \lambda_1, \ldots, \lambda_{\ell - 1}, \lambda_\ell \rangle ; q) = q \cdot R_k(\langle \lambda_1, \ldots, \lambda_{\ell - 1}, \lambda_\ell - 1 \rangle ; q) + {}\\ R_{k - 1}(\langle \lambda_1 - 1, \ldots, \lambda_{\ell - 1} - 1 \rangle ; q).
\end{multline}
(This is a special case of Dworkin's ``deletion-contraction" recurrence for $q$-rook numbers \cite[Thm.~6.10]{dworkin_interpretation_1998}.) Repeated applications of \eqref{eq:recurrence} yield
\begin{align*}
R_2(\langle 10, 9, 3, 2, 1\rangle; q) & = q \cdot R_2(\langle 10, 9, 3, 2\rangle; q) + R_1(\langle 9, 8, 2, 1\rangle; q), \\
R_2(\langle 10, 9, 3, 2\rangle; q) & = q^2 \cdot R_2(\langle 10, 9, 3\rangle; q) + (1 + q) \cdot R_1(\langle 9, 8, 2\rangle; q), \\ \intertext{and}
R_2(\langle 10, 9, 3\rangle; q) & = q^3\cdot R_2(\langle 10, 9\rangle; q) + (1 + q + q^2) \cdot R_1(\langle 9, 8\rangle; q).
\end{align*}
By \eqref{eq:full rank}, we have $R_2(\langle 10, 9\rangle; q) = ([9]_q)^2$, while the various instances of $R_1$ can be computed from \eqref{eq:rank 1}.  Plugging these values in and simplifying produces
\begin{multline*}
R_2(\langle 10, 9, 3, 2, 1\rangle; q) = q^{22} + 3q^{21} + 7q^{20} + 13q^{19} + 18q^{18} + 21q^{17} +
{}\\
22q^{16} + 20q^{15} + 21q^{14} + 
{}\\
20q^{13} + 17q^{12} + 12q^{11} + 5q^{10} + 4q^{9} + 3q^{8} + 2q^{7} + q^6,
\end{multline*}
which is manifestly non-unimodal.
\end{proof}

\begin{remark}
The example in Proposition~\ref{main prop} is the smallest partition that has a non-unimodal $q$-rook number: we checked by exhaustive computation (carried out in {\tt Sage} \cite{sagemath}, implemented with the simple recursion \eqref{eq:recurrence}, with base case $R_0(\lambda; q)=q^{|\lambda|}$) that if $|\lambda| \leq 24$, then the $q$-rook number $R_k(\lambda; q)$ is unimodal for all $k$.  The only other partitions of size $25$ with a non-unimodal $q$-rook number are rook equivalent to $\langle 10, 9, 3, 2, 1\rangle$.  For larger $n$, the number of partitions of $n$ with non-unimodal $q$-rook number in rank $2$ grows steadily; for example, there are $67$ rook equivalence classes of partitions of size $40$ with $R_2$ non-unimodal (out of $1113$ total classes).
\end{remark}

\begin{remark}
We record here the results of some further computational experiments:
\begin{itemize}
    \item The smallest partitions $\lambda$ for which $R_3(\lambda; q)$ is not unimodal are $\langle 14, 13, 5, 4, 3, 2, 1 \rangle$, of size $42$, and the partitions rook equivalent to it.  These partitions are also non-unimodal in rank $2$.
    \item The partitions $\lambda$ of size $43$ for which $R_3(\lambda; q)$ is not unimodal are $\langle 16, 12, 5, 4, 3, 2, 1 \rangle$ and the partitions rook equivalent to it.  These partitions \emph{are} unimodal in rank $2$.
%    \item There are no Ferrers boards of size $\leq 70$ that have \textcolor{blue}{non-unimodal} $q$-rook number in rank $4$.
    \item The smallest partitions $\lambda$ for which $R_4(\lambda;q)$ is not unimodal are $\langle 31, 17, 16, 5, 4, 3, 2\rangle$, of size $78$, and the partitions rook equivalent to it. These partitions are also non-unimodal in rank $2$ but are unimodal in rank $3$.
    \item  For the partition $\langle 69, 67, 65, 10, 9, 7, 6, 5, 4, 3, 2\rangle$ of size $247$, $R_5(\lambda; q)$ is not unimodal. This partition was found doing an exploratory search with code written with the help of {\tt Gemini} \cite{gemini}.  We have no reason to believe this is a minimum-size example, but we have not found a smaller one. See Table~\ref{table: non-unimodal examples} for similar non-unimodal examples for $k=6,7,8,9,10$. The code for these explorations is available on the {\tt arXiv} as an ancillary file and  a website with a visualization tool for  $q$-rook numbers is available at \url{https://ahmorales.math.uqam.ca/Animations/qrooks.html}.
    %For the partition $\langle 69, 67, 65, 10, 9, 7, 6, 5, 4, 3, 2\rangle$ of size $247$, $R_5(\lambda; q)$ is not unimodal; but we have no reason to believe this is a minimum-size example, but we have not found a smaller one.
\end{itemize}

\begin{table}[h]
\centering
\begin{tabular}{cp{11cm}r}
\hline
$k$ & Partition $\lambda$  & size \\ \hline
2 & $\langle10, 9, 3, 2, 1 \rangle$  & 25 \\ 
3 & $\langle 14, 13, 5, 4, 3, 2, 1\rangle$ & 42 \\ 
4 & $\langle 31, 17, 16, 5, 4, 3, 2 \rangle$ & 78\\ 
5 & $\langle 69, 67, 65, 10, 9, 7, 6, 5, 4, 3, 2 \rangle$ & 247 \\ 
%6 & 67, 67, 14, 14, 14, 14, 14, 14, 14, 14, 14, 14, 14, 13, 12, 11, 10, 9, 8, 7, 6, 5, 4, 3, 2, 1 & 379\\
6 & $\langle 67^2, 14^{11}, 13, 12, 11, 10, 9, 8, 7, 6, 5, 4, 3, 2, 1 \rangle$ & 379\\
7 & $\langle 78^2, 19^4, 18^6, 17, 16, 15, 14, 13, 12, 11, 10, 9, 8, 7, 6, 5, 4, 3, 2, 1 \rangle$  & 494\\ 
%7 & 78, 78, 19, 19, 19, 19, 18, 18, 18, 18, 18, 18, 17, 16, 15, 14, 13, 12, 11, 10, 9, 8, 7, 6, 5, 4, 3, 2, 1  & 494\\ 
8 & $\langle$122, 119, 31, 30, 29, 28, 27, 26, 25, 24, 23, 22, 21, 20, 19, 18, 17, 16, 15, 14, 13, 12, 11, 10, 9, 8, 7, 6, 5, 4, 3, 2, 1$\rangle$ & 785 \\
9 & $\langle216^2, 39, 38^2, 37, 34^4, 33^2, 32^2, 30^4, 29, 28, 24^2, 22, 20^2, 19, 18$, $15^4, 13^2, 1^4\rangle$ & 1176  \\ 
%9 & $\langle$216, 216, 39, 38, 38, 37, 34, 34, 34, 34, 33, 33, 32, 32, 30, 30, 30, 30, 29, 28, 24, 24, 22, 20, 20, 19, 18, 15, 15, 15, 15, 13, 13, 1, 1, 1, 1$\rangle$ & 1176  \\ 
10 & $\langle244^2, 66, 65, 62, 59, 57^3, 56^4, 54, 43, 42^2, 40, 39, 35, 34^2, 21^3, 20^5,$  $18,16^2, 13^2, 12^4, 9^5, 1^{14}\rangle$ & 1801 \\ \hline
\end{tabular}

\caption{Examples of partitions $\lambda$ such that $R_k(\lambda;q)$ is not unimodal for $k=2,\ldots,10$. The examples for $k=2,3,4$ are minimum-size examples, while those for $k=5,\ldots,10$ were found doing a exploratory search with code written with the help of {\tt Gemini} \cite{gemini} based on the heuristic of {\tt PatternBoost} \cite{PatternBoost}. }
\label{table: non-unimodal examples}

\end{table}

%It seems reasonable to guess that the failure to find an example of non-unimodality in rank $4$ is simply that the smallest examples have size larger than $70$, and furthermore to conjecture that for each integer $k > 1$, there is some partition $\lambda$ for which $R_k(\lambda; q)$ is not unimodal.  
It seems reasonable to conjecture that for each integer $k > 1$, there is some partition $\lambda$ for which $R_k(\lambda; q)$ is not unimodal.
Typically, it seems that when $R_k(\lambda; q)$ is not unimodal, there is a gap between $k$ and the number of parts $\ell(\lambda)$ of $\lambda$.  For example, every partition $\lambda$ with less than $100$ boxes for which $R_2(\lambda; q)$ is non-unimodal has $\ell(\lambda) \geq 5$.  The smallest example with a gap of size $2$ is for the non-unimodal polynomial $R_3(\langle 24, 15, 14, 3, 2 \rangle; q)$.  Is $R_{\ell-1}(\lambda;q)$ always unimodal?
\end{remark}
\begin{remark}
In the tables of data  they collected in \cite[p.~273--274]{GarsiaRemmel}, Garsia and Remmel noted that the coefficients of the \defn{total $q$-rook number} \[
R(\lambda; q) := \sum_{k \geq 0} R_k(\lambda; q)
\] 
also appeared to be unimodal. We have verified this for all partitions $\lambda$ such that $|\lambda| \leq 70$. For the staircase $\delta_n$, $R(\delta_n;q)$ is the \emph{$q$-Bell number} \cite{wagner2004partition}. Note that $R(\lambda;q)$ is not necessarily log-concave since $R(\langle 2,1\rangle;q)= 1+2q+q^2+q^3$.
\end{remark}

\begin{remark} \label{rem:q stir 1}
The $q$-Stirling numbers $S_q(n,k)$ of the second kind, defined by $S_q(n, k):=R_{n-k}(\delta_n;q)$, or alternatively by the recurrence
\begin{equation} \label{eq:recurence q Stirling}
S_q(n,k) =  q^{k-1} S_q(n-1,k-1) + [k]_q S_q(n-1,k)  \quad\text{for}\quad 0<k \leq n, \qquad S_q(0,k) = \delta_{0,k},
\end{equation}
were further conjectured to be log-concave by Wachs--White \cite[p.~45]{WACHS199127}. This conjecture was discussed more recently in work of Sagan and Swanson \cite[Conj.~7.4]{SAGAN2024103899}, who verified it for $1\leq k\leq n \leq 50$. We have checked it  for $1\leq k \leq n \leq 250$. See \cite[\S 1.2]{SAGAN2024103899} for a history of these polynomials.

It follows easily from the recurrence~\eqref{eq:recurence q Stirling} that these $q$-Stirling numbers can be obtained from the following evaluation of a {\em complete symmetric polynomial} \cite{Gould}:
\[
S_q(n,k) = q^{\binom{k}{2}}h_{n-k}([1]_q,[2]_q,\ldots,[k]_q).
\]
Such complete symmetric functions are instances of Schur functions: $h_{n-k}(x_1,\ldots,x_k) = s_{n-k}(x_1,\ldots,x_k)$. 
Could the evaluation above be useful to show log-concavity or unimodality of $S_q(n,k)$, as in the {\em  principal specialization} \cite[\S 7.8]{EC2} of Schur functions \cite[Thm.~13]{StanleyLogConcave}, or by exploiting the {\em denormalized Lorentzian property} of Schur polynomials \cite{HMMStD}?
\end{remark}

\begin{remark}
For all partitions $\lambda$ of size $|\lambda|\leq 75$ and all $k$, the rank-$k$ $q$-rook number $R_k(B_{\lambda};q)$ is log-concave whenever $R_1(B_{\lambda};q)$ is log-concave (i.e., if the sequence of diagonal lengths $d_i(\lambda)$ are log-concave).  Could this be true for all $\lambda$?  Note that $R_1(\delta_n;q)$ is log-concave, so this is a strengthening of the Wachs--White conjecture mentioned in the previous remark.    
\end{remark}

\begin{remark} \label{rem:q stir 2}
In \cite[Cor.~2.6]{SAGAN1992289}, Sagan showed that the $q$-Stirling numbers of the second kind satisfy the following alternative notion of log-concavity:\footnote{See also \cite{LEROUX199064} for an analogous result for the shifted polynomials $q^{-\binom{k}{2}}S_q(n,k)$.}
\[
S_q(n,k)^2-S_q(n,k-1)S_q(n,k+1) \in \mathbb{N}[q].
\]
Based on calculations for partitions $\lambda$ such that  $|\lambda|\leq 50$, we conjecture that the $q$-rook numbers $R_k(\lambda;q)$ satisfy this notion of log-concavity, i.e., that
\[
R_k(\lambda; q)^2-R_{k - 1}(\lambda; q)R_{k + 1}(\lambda; q) \in \mathbb{N}[q]
\]
for all $k$ and $\lambda$.  At $q = 1$, this conjecture is a consequence of \cite[Thm.~1]{Nijenhuis}; see \cite{HOW} for more about log-concavity for the $q = 1$ case. 
\end{remark}

\subsection*{Acknowledgements}

The second author thanks Jim Haglund for first telling us and Bruce Sagan for reminding us about the question of the unimodality of the Garsia--Remmel $q$-rook numbers. We also thank Jonathan Leake for helpful discussions on Lorentzian polynomials.

The first author thanks Swee Hong Chan, Hong Chen, Guilherme Zeus Dantas e Moura, and Elizabeth Mili\'cevi\'c for their scintillating dinner conversation at the Mid-Atlantic Algebra, Geometry, and Combinatorics Workshop (MAAGC 2024).  In particular, we thank Swee Hong for pointing out (in response to a remark by the first author) that although it is obviously \emph{possible} to compute many thousands or millions of $q$-rook numbers to test the conjecture, it is not clear whether anyone \emph{actually had done so}.

The first author was partially supported by Simons Foundation gift
MPS-TSM-00006960.  The second author acknowledges the support of the Natural Sciences and Engineering Research Council of Canada (NSERC) funding reference number RGPIN-2024-06246, and was partially supported by the NSF grant DMS-2154019.

Lastly, the authors have no relevant financial or non-financial competing interests to report.

\printbibliography

% \section{Musings}

% The $q$-Stirling number $S_q(n,k)$ can be viewed as an evaluation of a complete symmetric function:
% \[
% S_q(n,k) = q^{\binom{k}{2}}h_{n-k}([1]_q,[2]_q,\ldots,[k]_q).
% \]
% This follows from the recurrence for complete symmetric polynomials
% \[
% h_{n-k}(x_1,\ldots,x_k) = h_{n-k}(x_1,\ldots,x_{k-1}) + x_k h_{n-k-1}(x_1,\ldots,x_k),
% \]
% evaluating at $x_i = [i]_q = 1+q+\cdots + q^{i-1}$ and adjusting the power of $q$ (see \cite[Eq. (3)]{WACHS199127}).
% Complete symmetric function is an instance of a Schur function $h_{n-k}(x_1,\ldots,x_k)=s_{n-k}(x_1,\ldots,x_k)$. Schur functions are normalized Lorentzian \cite{HMMStD}, and Lorentzian polynomials are preserved under {\em linear} substitutions of variables with nonnegative coefficients. 

\end{document}